\documentclass{amsart}
\usepackage{cite,graphicx}

\usepackage[pdftex,plainpages=false,hypertexnames=false,pdfpagelabels]{hyperref}
\usepackage{xcolor}
\definecolor{dark-red}{rgb}{0.7,0.25,0.25}
\definecolor{dark-blue}{rgb}{0.15,0.15,0.55}
\definecolor{medium-blue}{rgb}{0,0,0.65}
\hypersetup{
  colorlinks, linkcolor={purple},
  citecolor={medium-blue}, urlcolor={medium-blue}
}

 \newcommand\PA{\mathcal{P}'}
  \newcommand\C{\mathbb{C}}

\newtheorem{thm}{Theorem} 
\newtheorem{defn}{Definition} 
\newtheorem{ex}{Example} 
\newtheorem{fig}{Figure}
\newtheorem{lem}{Lemma}

\begin{document}
\bibliographystyle{amsplain}
\title{A Diagrammatic Multivariate Alexander Invariant of Tangles}
\author{K. Grace Kennedy}
\email{kgracekennedy@math.ucsb.edu}
%\email{bigelow@math.ucsb.edu, kgracekennedy@alumni.sewanee.edu}
\maketitle

\begin{abstract}
Recently, Bigelow defined a diagrammatic method for calculating the Alexander polynomial of a knot or link by resolving crossings in a planar algebra.  I will present my multivariate version of Bigelow's calculation. The advantage to my algorithm is that it generalizes to a multivariate tangle invariant up to Reidemeister I.  I will conclude with a possible link to subfactor planar algebras from the work of Jones and Penneys.
\end{abstract}

%%%%%%%%%%%%%%%%%%%%%%%%%%%%%%%%%%%%%%%%%%%%%%%%%
\section{Introduction}
\label{Introduction}

Ever since their introduction in Jones's \emph{Planar Algebras I} \cite{JonesPA}, planar algebras have been linked to knot invariants.  Here we present a planar algebra that we will use to describe the multivariate Alexander polynomial.  It is the same planar algebra that Bigelow used for the single variable Alexander polynomial \cite{Bigelow}.  Alexander discovered what would become known as the Alexander polynomial of a knot and published it in his 1928 paper ``Topological Invariants of Knots and Links."  He included the description of a skein relation under Section 12 ``Miscellaneous theorems" \cite{Alex}.  In 1969, Conway rediscovered the skein relation
$$\begin{matrix}{}
\includegraphics[scale=.4]{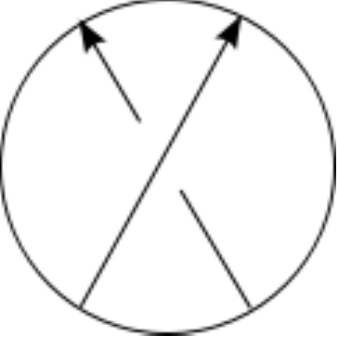}
\end{matrix}
- \begin{matrix}{}
\includegraphics[scale=.4]{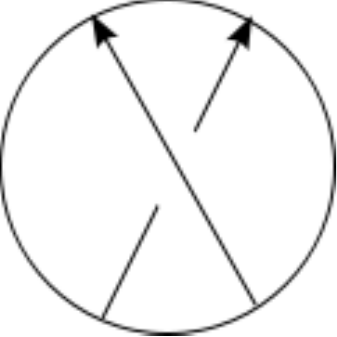}
\end{matrix}=(q-q^{-1}) \begin{matrix}{}
\includegraphics[scale=.4]{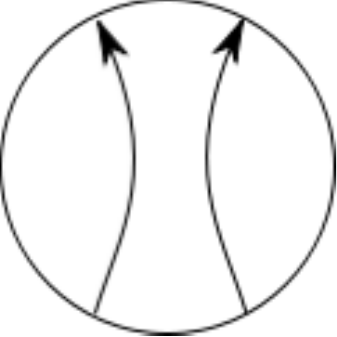}
\end{matrix}
\hskip.03in,
$$
which could be used for both identifying when a knot invariant was equivalent to Alexander's and for calculating the invariant by making local changes to crossings.  A year later, Conway discovered the Conway potential function, or the multivariate Alexander polynomial \cite{Con1970}, which is more effective at distinguishing links than the single variable invariant.  In 1993, Murakami published a list of axioms for the multivariate Alexander polynomial in his paper ``A State Model for the Multi-variable Alexander Polynomial" \cite{JMur}.  This contribution was analogous to the discovery of the skein relation for the single variable version.  One could then determine if a multivariate link invariant was the one defined by Conway and could also calculate the invariant using the axioms to make local changes.

Archibald generalized the multivariate Alexander polynomial to a tangle invariant and published it in her thesis \cite{ArchThesis}.  In 2010, Bigelow presented a single variable, diagrammatic tangle invariant that is also a generalization of the Alexander polynomial \cite{Bigelow}.  For his method, one considers oriented knots and links as 1-tangles with the single unclosed strand having endpoints on the boundary of a disk.  This represents the knot or link formed by the closure of the 1-tangle.  The knot invariant is calculated by sending knots and links as 1-tangles into a certain planar algebra.  The algorithm is not specific to 1-tangles, so it generalizes easily to tangles with more than one unclosed strand.  In this paper, we present a multivariate version of Bigelow's algorithm.  There is strong evidence to suggest that our multivariate tangle invariant is the same as the one defined by Archibald \cite{ArchThesis}.

Definitions of the planar algebra we use and the tangle planar algebra are in Section \ref{PA}.  The planar algebra that we will define is the Motzkin planar algebra defined by Jones \cite{VJMotzkin}.  The Motzkin algebra was also recently studied by Benkart and Halverson in \cite{MA} and is related to the planar rook algebra in Flath, Halverson, and Herbig \cite{PRA} and Bigelow, Ramos, and Yi \cite{AJP}.  The goal of Section \ref{PA} is to give the reader a good understanding of an unshaded planar algebra in order to follow the rest of the paper or to go on to read about shaded planar algebras and subfactor planar algebras.  Not all of the background is essential for following the rest of the paper but is included to give some context to the planar algebra presented here.  For more on shaded planar algebras and subfactors, see \cite{JonesPA}.  References \cite{VJMotzkin} and  \cite{FreeProbGJS} each provide a good, condensed explanation of shaded planar algebras.

Section \ref{Algorithm} will be a presentation of a multivariate version of Bigelow's algorithm for the Alexander polynomial.  Sections \ref{TI} and \ref{MVAP} will be dedicated to verifying that this is indeed a knot invariant and in fact the same multivariate knot invariant defined by Conway in 1970.  In Section \ref{Conc}  we conclude with how to extend this multivariate Alexander polynomial to an invariant of tangles and some ideas for future research, including a connection to infinite index subfactors from the work of Jones and Penneys \cite{Dave}.

I would like to thank my advisor, Professor Stephen Bigelow for his help and guidance and also the math department at the University of California, Santa Barbara.  Thank you to David Penneys for pointing out how this research relates to subfactor theory.  Thank you to Michael Polyak for pointing me to his paper on a minimal generating set of Reidemeister moves, which shortened the proof that our algorithm is a link invariant.  And thank you to Cardiff University and the Marie Curie Training Network for funding during the write up of these results.

%%%%%%%%%%%%%%%%%%%%%%%%%%%%%%%%%%%%%%%%%%%%%%%%%
\section{Three planar algebras}
\label{PA}

Shaded planar algebras were originally introduced by Vaughan Jones to study subfactors \cite{JonesPA}, although there are several papers that give definitions of planar algebras that do not involve functional analysis  \cite{MPS}, \cite{FreeProbGJS}.  Here we only need the unshaded definition given in \cite{MPS}.

First, we define a planar tangle or a planar tangle diagram.  A \emph{planar tangle} is a disk with endpoints marked along the boundary and $r\geq 0$ internal disks each with $k_l, l=1,2,\ldots r$ endpoints.  Strings with no crossings connect the endpoints marked on the disks.  There is also a marked point on the boundary of each internal and external disk.  For instance, Example \ref{ExPT} provides two examples of planar tangles.

\begin{ex}
\label{ExPT}
$$
\begin{matrix}{}
\includegraphics[scale=.4]{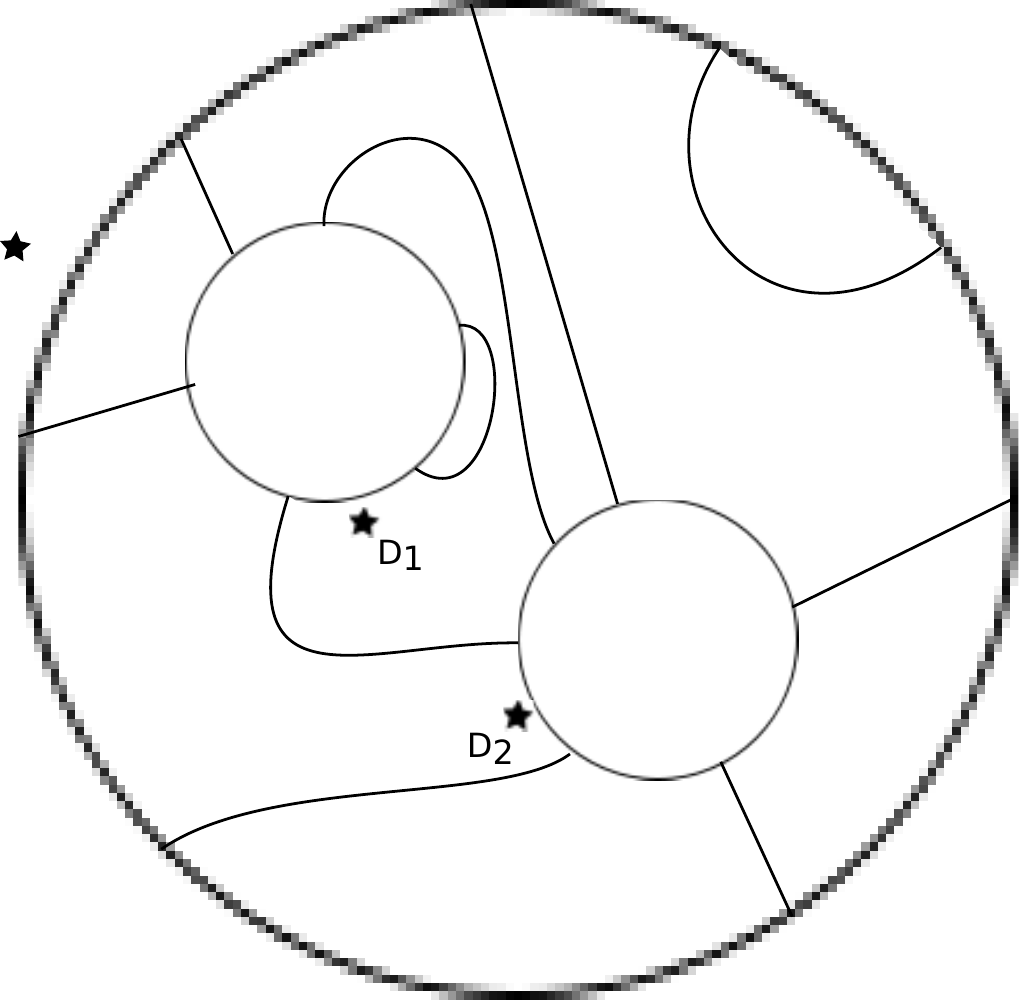}
\end{matrix}
\hskip.1in
\textrm{ or }
\hskip.1in
\begin{matrix}{}
\includegraphics[scale=.4]{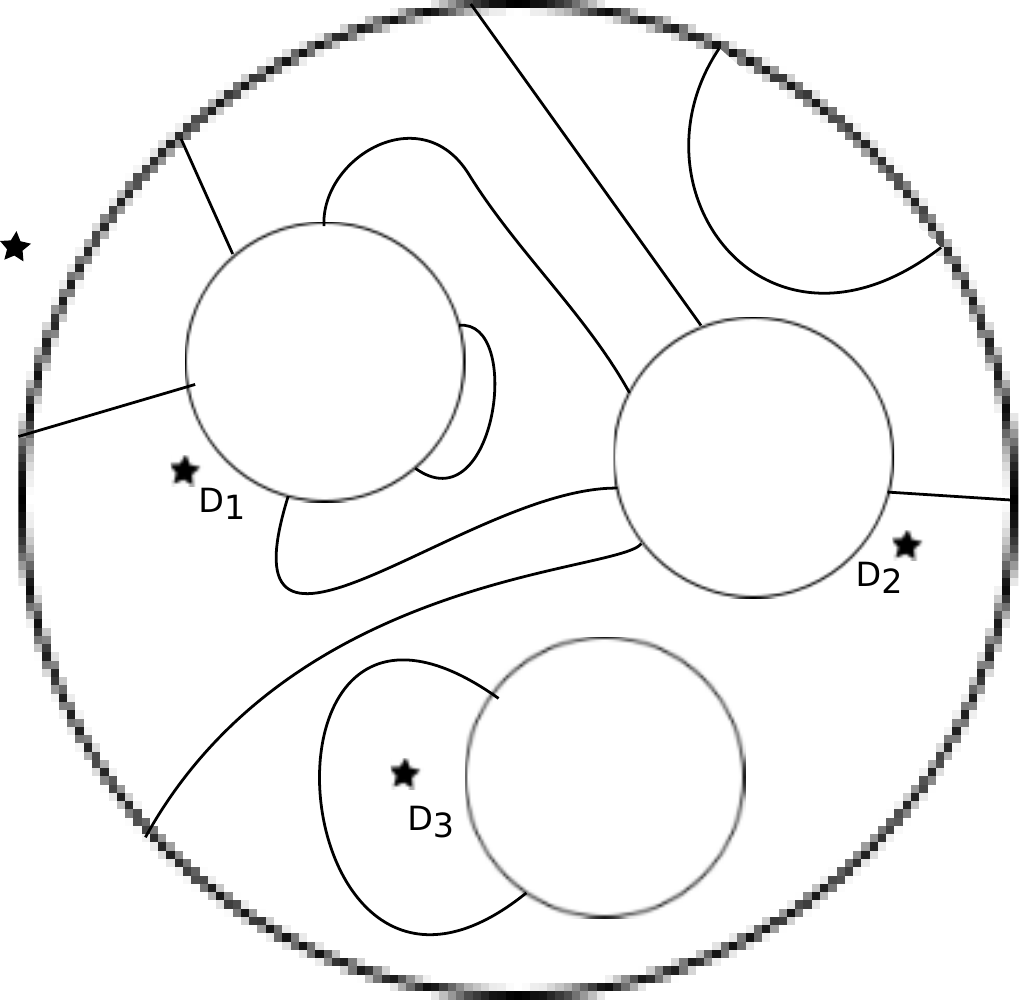}
\end{matrix}
$$
\end{ex}

\begin{defn}\label{DefPA}
An \emph{unshaded planar algebra} $\mathcal{P}$ is a sequence of vector spaces $\{P_k\}_{k=0}^{\infty}$ along with an assignment for every planar tangle to a linear map between tensor products of the vector spaces.  A planar tangle with $K$ endpoints on the outer boundary of the diagram and $r\geq 0$ internal disks with $k_l, l=1,2,\ldots r$ endpoints on the boundaries represents a linear map from $P_{k_1}\otimes\ldots\otimes P_{k_l}$ to $P_K$.  Composition of maps is associative as long as the maps are composable.\end{defn}

The set of planar tangles defines the \emph{planar operad}.  The first planar tangle in Example \ref{ExPT} represents a linear map from $P_6\otimes P_6$ to $P_8$.  The second planar tangle represents a linear map from $P_6\otimes P_5\otimes P_2$ to $P_7$.  It is frequently helpful to think of the basis vectors of the various vector spaces as pictures that can be glued into the planar tangles.  The Temperley-Lieb planar algebra is a good example to keep in mind.

\begin{ex}
\label{ExTLPA}
The \emph{Temperley-Lieb planar algebra} has vector spaces $P_{2n}=TL_{2n}$, the Temperley-Lieb algebra with $n-$strands, and $P_{2n+1}=0$ for $n\geq 0$.  The planar tangles act on the basis vectors of each $TL_{2n}$ by gluing the diagrams into the tangles, lining up the marked points and strands, and erasing the boundary of each internal disk.
\end{ex}

Most people first see the Temperley-Lieb diagrams drawn in boxes that can be stacked on top of one another like the braid group \cite{West}, \cite{Abram}.  The diagrams in $TL_{2n}$ have $2n$ endpoints on the boundary of a disk (or box) and $n$ strands that connect the endpoints without crossings up to isotopy.  A closed strand can be deleted at the expense of scaling the diagram by a constant.

Stacking Temperley-Lieb diagrams on top of one another can also be done in planar algebras using the multiplication tangle:
$$
\begin{matrix}{}
\includegraphics[scale=.4]{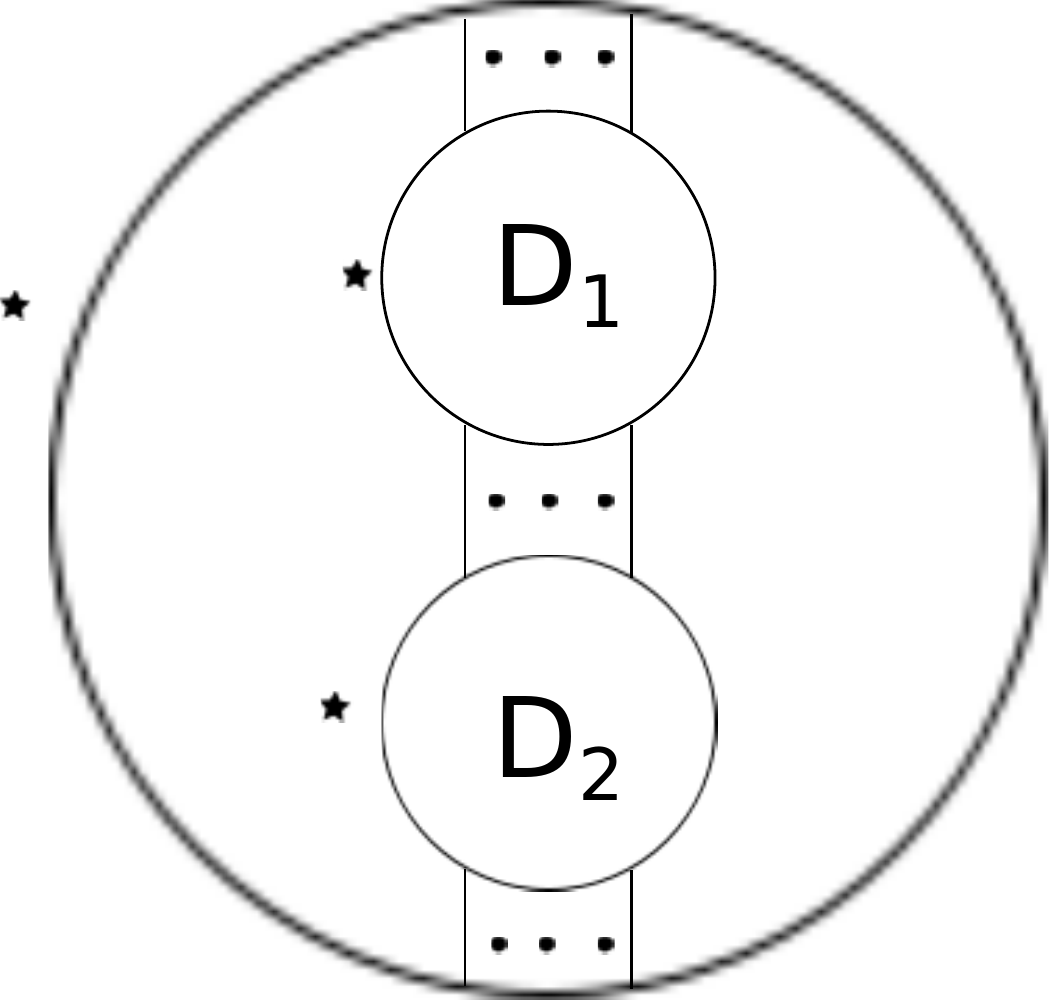}
\end{matrix}
\hskip.03in.
$$

In a general planar algebra, the multiplication tangle is an operation from $P_{2n}\otimes P_{2n}$ to $P_{2n}$.  There is also a dual tangle and a trace tangle (below respectively):

$$
\begin{matrix}{}
\includegraphics[scale=.4]{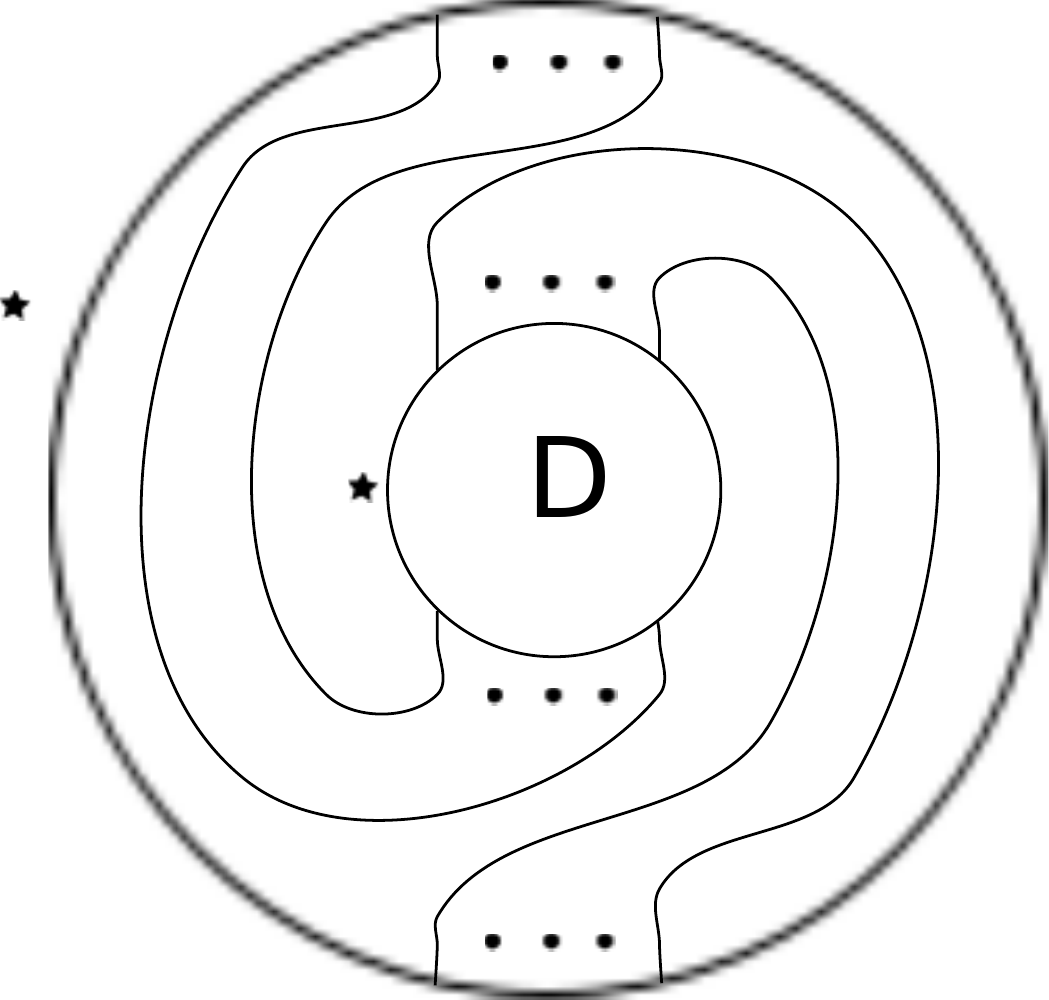}
\end{matrix}
\hskip.1in
\textrm{ and }
\hskip.1in
\begin{matrix}{}
\includegraphics[scale=.4]{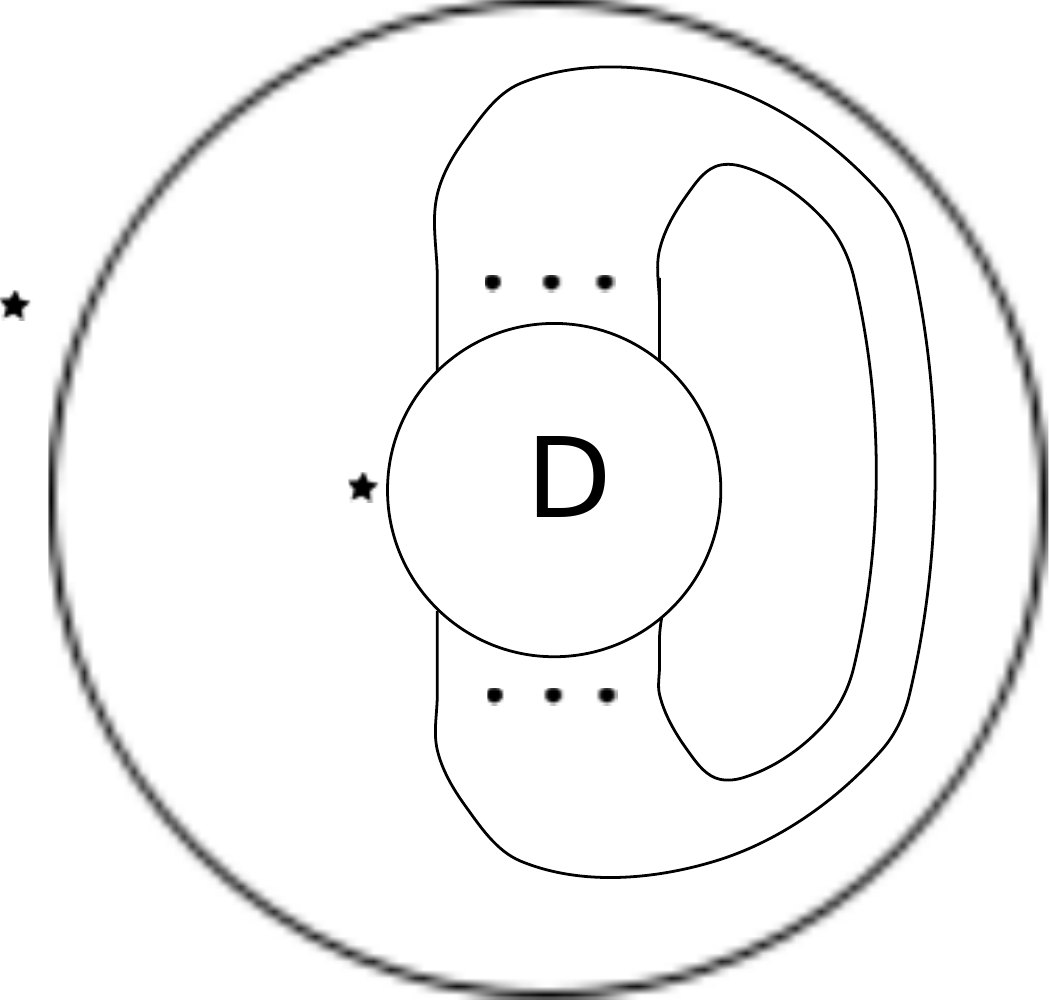}
\end{matrix}
\hskip.03in.
$$
The first planar tangle is a map from $P_{2n}$ to itself, which corresponds to rotation by $\pi$ radians.   The second is a map from from $P_{2n}$ to $P_0$.

If $P_0$ is one dimensional, then the trace tangle is a map into $\C$, and together with an adjoint equivalent to horizontal flipping can be used to define a sesquilinear form.  In the special case of a subfactor planar algebra, this form must be positive definite.  Also for a subfactor planar algebra, all of the vector spaces must be finite dimensional, and the odd numbered vector spaces must be zero.  Finally, subfactor planar algebras must be \emph{spherical}, that is to say
$$
\begin{matrix}{}
\includegraphics[scale=.4]{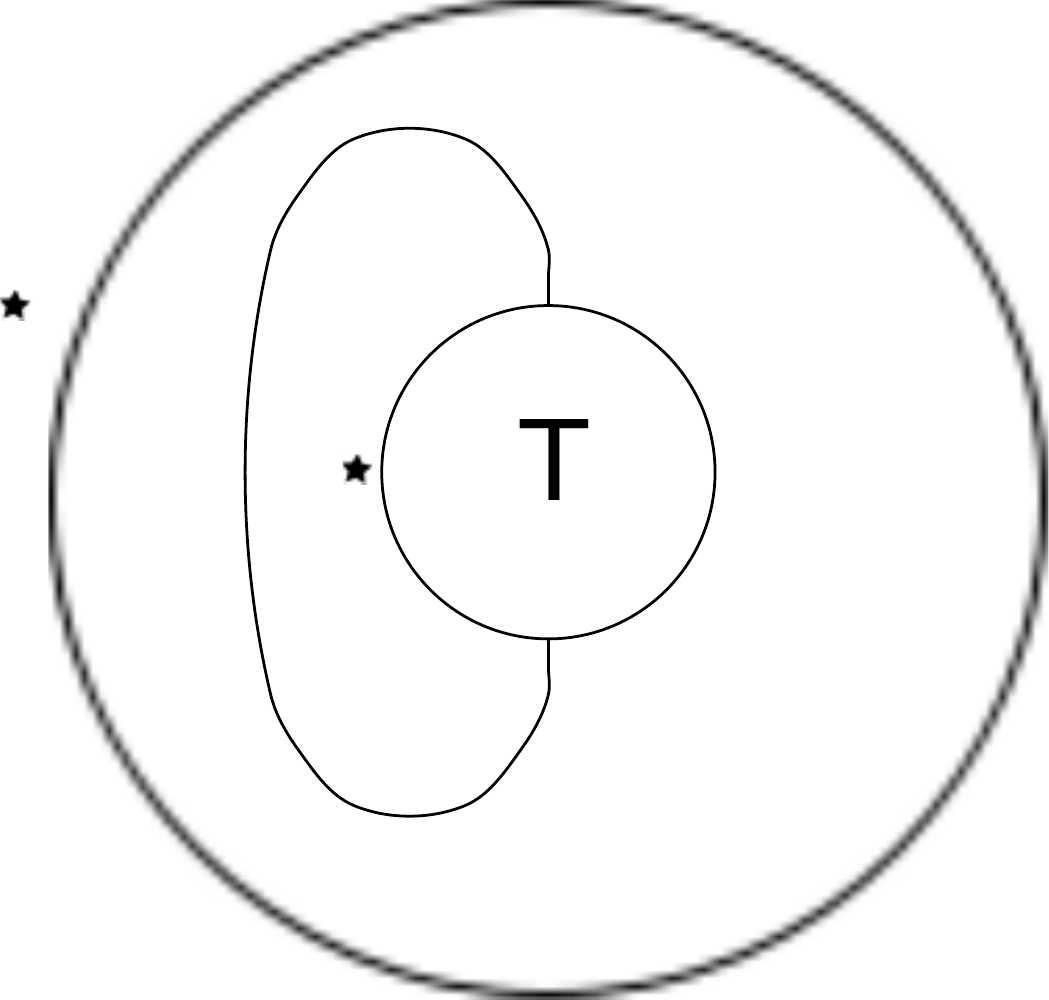}
\end{matrix}
\hskip.1in
=
\hskip.1in
\begin{matrix}{}
\includegraphics[scale=.4]{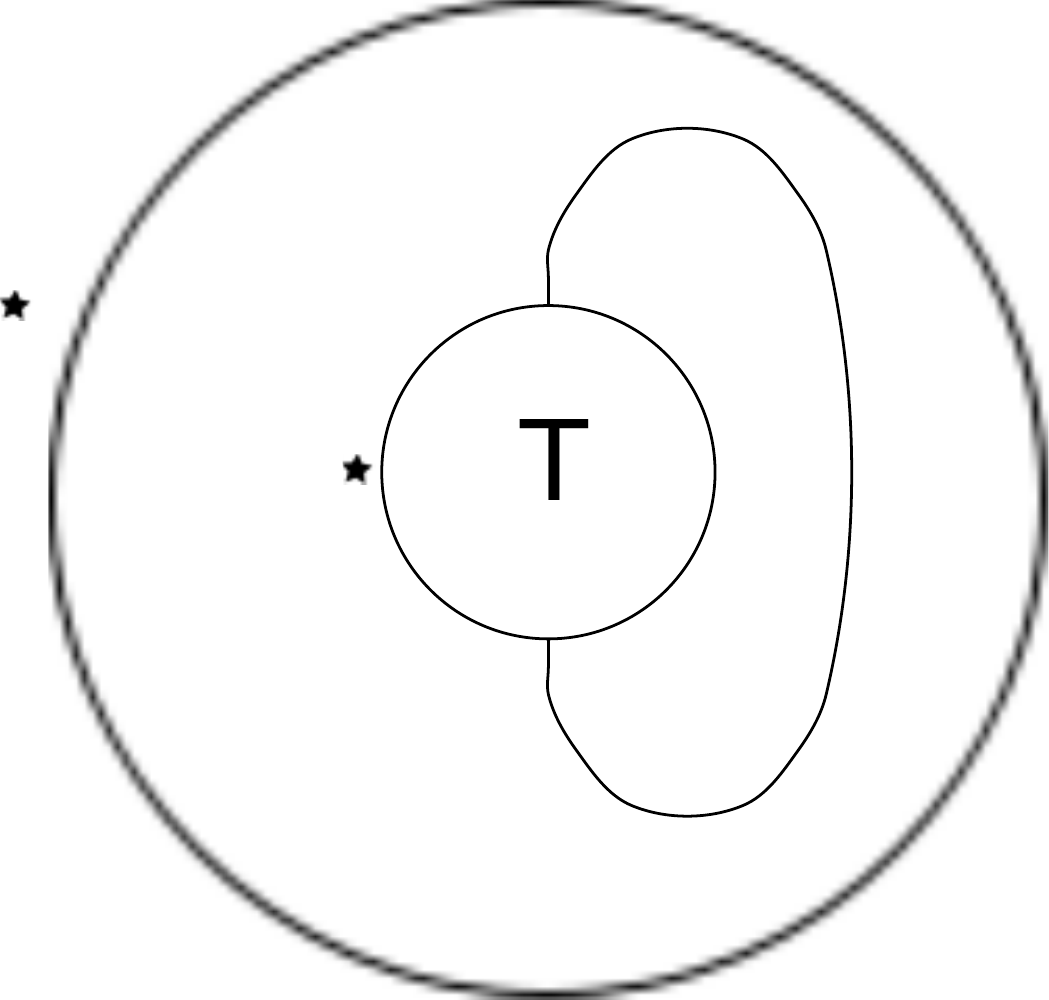}
\end{matrix}
\hskip.03in.
$$

As subfactor planar algebras are shaded, they only have nonzero vector spaces that glue into diagrams with an even number of endpoints on the boundary.  Papers that only include diagrams with an even number of endpoints frequently take $P_n$ to be the vector space with $2n$ endpoints on the boundary.

The planar algebra generated by an element belonging to a specific vector space, $P_k$, has vector spaces that contain all linear combinations of all of the diagrams that can be generated by the planar operations.  Below we will see two examples of such planar algebras.  The planar tangles will always generate diagrams that look like the Temperley-Lieb diagrams.  If there is the relation that a closed loop is equivalent to a constant times the empty diagram, there is a copy of the Temperley-Lieb algebra in the planar algebra.

\begin{ex}
\label{ExTanPA}
The \emph{tangle planar algebra}, $\mathcal{T}$ is defined similarly to the Temperley-Lieb planar algebra except that we allow crossings.  The vector spaces are $P_{2n}=T_{2n}$, the tangle algebra with $n$ unclosed strands, and $P_{2n+1}=0$ for $n\geq 0$.  The planar tangle diagrams act on the basis vectors of each $T_{2n}$ by gluing tangles into the internal disks, lining up the marked points and strands, then erasing the boundary of each internal disk.
\end{ex}

One could think of the tangle planar algebra as being the planar algebra generated by two elements in $P_4$, the positive and negative crossings, modulo the Reidemeister moves.  Adding skein relations to these equivalences renders evaluating a closed diagram equivalent to calculating a knot invariant \cite{JonesPA}.

\begin{defn}
Let $\PA$ be the planar algebra generated by a single element in $P_1$:
$$
\begin{matrix}{}
\includegraphics[scale=.4]{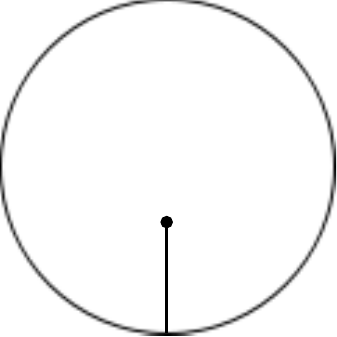}
\end{matrix}
$$
with relations
 in $P_0$ and $P_4$ respectively
$$
\begin{matrix}{}
\includegraphics[scale=.4]{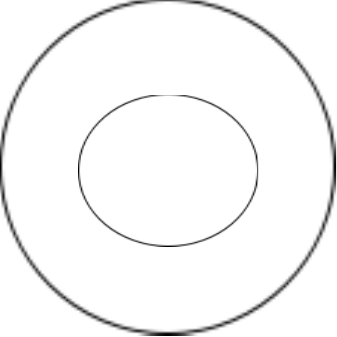}
\end{matrix}=0,
\hskip.1in
\begin{matrix}{}
\includegraphics[scale=.4]{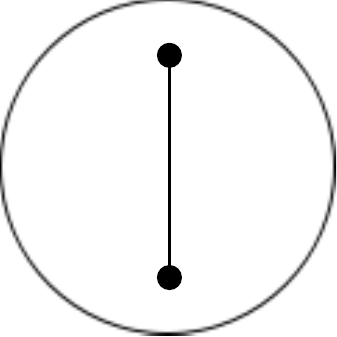}
\end{matrix}
=
\begin{matrix}{}
\includegraphics[scale=.4]{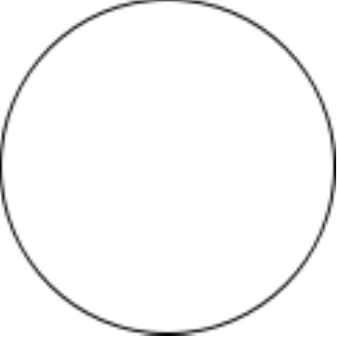}
\end{matrix}, \textrm{ and }
\hskip.1in
\begin{matrix}{}
\includegraphics[scale=.4]{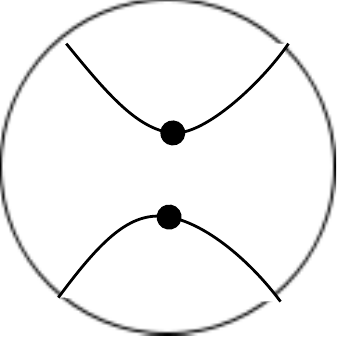}
\end{matrix}
+
\begin{matrix}{}
\includegraphics[scale=.4]{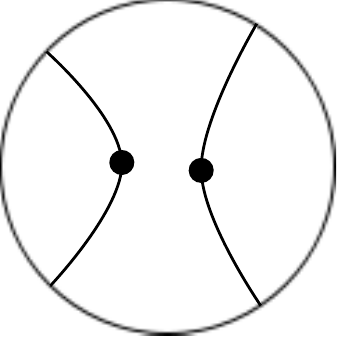}
\end{matrix}
=0.
$$
To ease notation, these relations make use of the definition of the dotted strand in Definition \ref{DefDottedStrand} (below).
\begin{defn}
\label{DefDottedStrand}
Define the dotted strand in $P_2$ as
$$\begin{matrix}{}
\includegraphics[scale=.4]{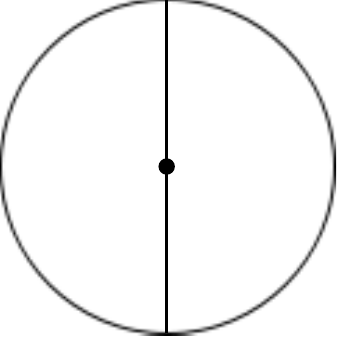}
\end{matrix}
=
\begin{matrix}{}
\includegraphics[scale=.4]{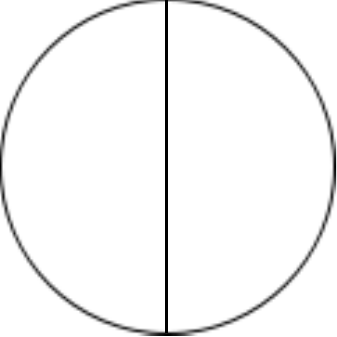}
\end{matrix}
-
\begin{matrix}{}
\includegraphics[scale=.4]{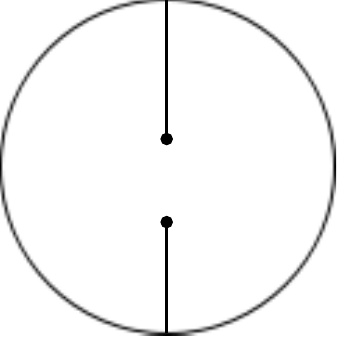}
\end{matrix}
\hskip.03in.
$$
\end{defn}
\end{defn}

So the vector space $P_k$ has basis vectors indexed by disks with $k$ endpoints on the boundary and strings with no crossings and one or two endpoints on the boundary.  For the basis vectors, we require all strands with two endpoints on the boundary be dotted.  Requiring that there be no dots gives another basis.  The vector spaces $P_k$ are finite dimensional.  Since a closed loop is equal to zero, $P_0$ is 1-dimensional with every non-zero element equal to a multiple of the empty diagram.  However, the inner product is not positive definite, and $P_1$ being non-zero lets us know immediately this is not a subfactor planar algebra.

The definition of the dotted strand is included to make the last relation less cumbersome to write.  It would otherwise include eight terms.  This notation gives the following two relations given in Bigelow's Lemma 3.4 \cite{Bigelow}:

\begin{lem}
\label{LemBig3.4}
Where the internal dot is as in Definition \ref{DefDottedStrand}, we have the following two relations:
$$\begin{matrix}{}
\includegraphics[scale=.4]{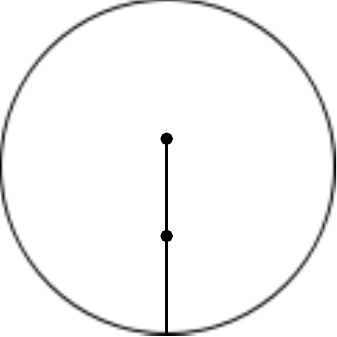}
\end{matrix}
=0
$$
and
$$
\begin{matrix}{}
\includegraphics[scale=.4]{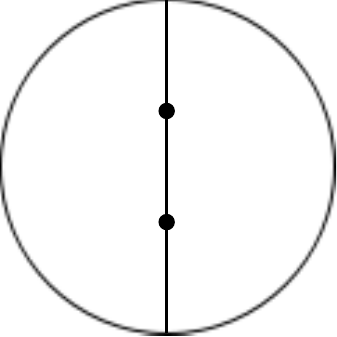}
\end{matrix}
=
\begin{matrix}{}
\includegraphics[scale=.4]{DottedStrand.pdf}
\end{matrix}
\hskip.03in.
$$
\end{lem}

The planar algebra $\PA$ was defined by Bigelow to give the diagrammatic algorithm for the Alexander polynomial in \cite{Bigelow}.  It was also defined in Halverson and Benkart's ``Motzkin Algebras" in \cite{MA} and as a specific example of a planar algebra by Jones in \cite{VJMotzkin}.  It has also come up in the theory of infinite index subfactors in the work of Jones and Penneys\cite{Dave}, \cite{DaveInfInd}, which we will briefly discuss in Section \ref{Conc}.  In the following section, we will outline how to send tangles into this planar algebra and retrieve the multivariate Alexander polynomial.

%%%%%%%%%%%%%%%%%%%%%%%%%%%%%%%%%%%%%%%%%%%%%%%%%
\section{Algorithm and theorem}
\label{Algorithm}

We will think of oriented knots and links as existing in the second vector space of the tangle planar algebra, the vector space of tangles with only one unclosed strand.  Each component has an orientation and is assigned a color, or variable.  The knot or link we are considering is the one we get from making the obvious closure.

\begin{fig}
\label{FigTE}
$$
\includegraphics[scale=.6]{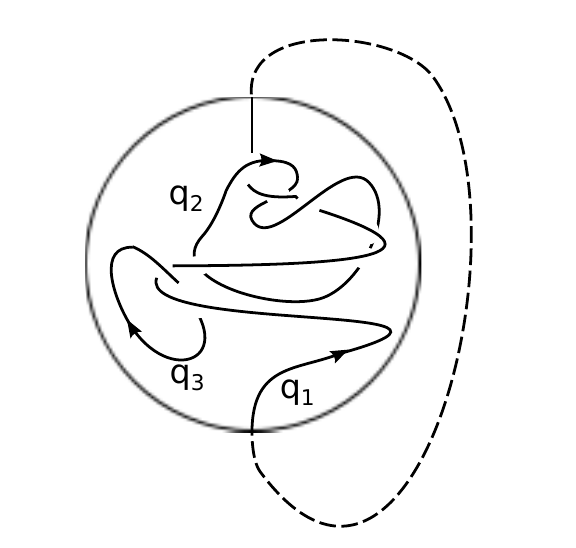}
$$
This is an example of a three-component link with colors $q_1$, $q_2$, and $q_3$.
\end{fig}

\begin{defn}
\label{DeltaM}
Let $\mathbb{T}$ be the set of oriented tangles written like elements in the tangle planar algebra with loose strands having endpoints on the boundary of a disk.  Define $\Delta_m:\mathbb{T}\rightarrow\PA$ to be the map from oriented tangles to $\PA$ that resolves positive and negative crossings as follows:
%%%Pos Crossing Resolution%%%
$$
\begin{matrix}{}
\includegraphics[scale=.4]{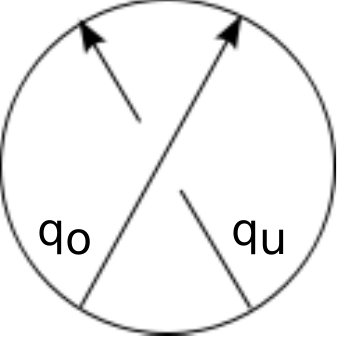}
\end{matrix} =
q_o\begin{matrix}{}
\includegraphics[scale=.4]{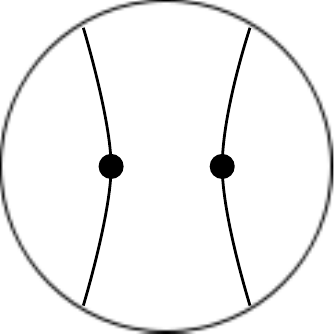}
\end{matrix} 
+
q_o\begin{matrix}{}
\includegraphics[scale=.4]{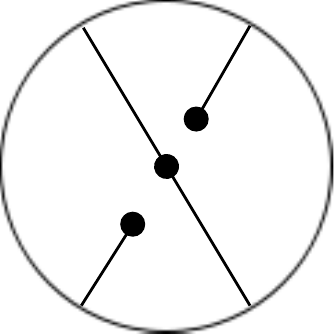}
\end{matrix} 
+
(q_u-q_u^{-1})\begin{matrix}{}
\includegraphics[scale=.4]{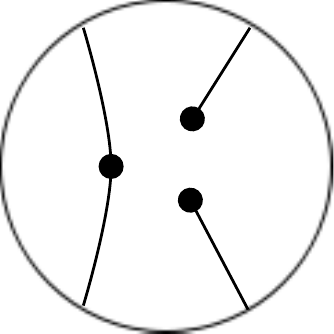}
\end{matrix} 
+
q_o^{-1}\begin{matrix}{}
\includegraphics[scale=.4]{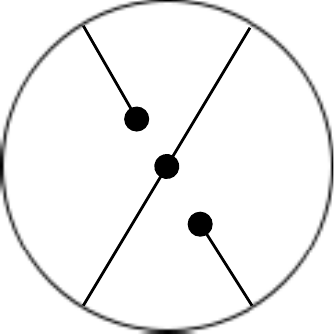}
\end{matrix} 
-
q_o^{-1}\begin{matrix}{}
\includegraphics[scale=.4]{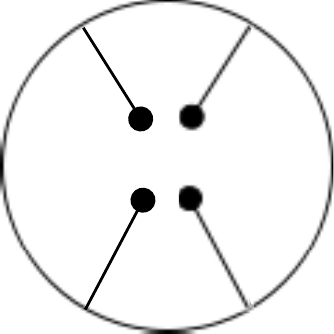}
\end{matrix} 
$$
%%%Neg Crossing Resolution%%%
$$\begin{matrix}{}
\includegraphics[scale=.4]{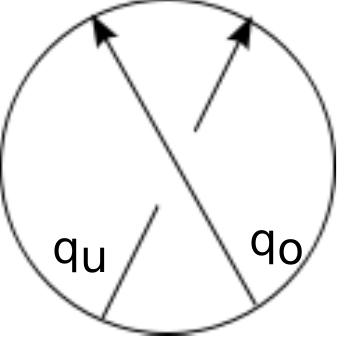}
\end{matrix}=
q_o^{-1}\begin{matrix}{}
\includegraphics[scale=.4]{DBasis1423.pdf}
\end{matrix} 
+
q_o^{-1}\begin{matrix}{}
\includegraphics[scale=.4]{DBasis13cc.pdf}
\end{matrix} 
-
(q_u-q_u^{-1})\begin{matrix}{}
\includegraphics[scale=.4]{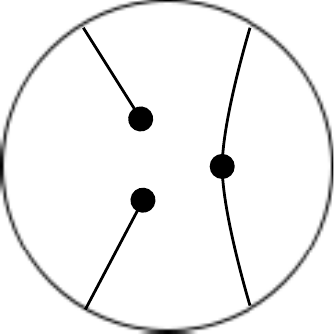}
\end{matrix} 
+
q_o\begin{matrix}{}
\includegraphics[scale=.4]{DBasis24cc.pdf}
\end{matrix} 
-q_o
\begin{matrix}{}
\includegraphics[scale=.4]{DBasis0.pdf}
\end{matrix} 
\hskip.03in.
$$
\end{defn}

When we let the same variable be associated to every component then we have Lemmas 5.1-5.4 from \cite{Bigelow}:

\begin{lem}
\label{LemRM}
The map $\Delta_m$ on a knot or link where the same variable, $q$, is associated to each component satisfies the skein relation
$$\begin{matrix}{}
\includegraphics[scale=.4]{PosCross.pdf}
\end{matrix}
- \begin{matrix}{}
\includegraphics[scale=.4]{NegCross.pdf}
\end{matrix}=(q-q^{-1}) \begin{matrix}{}
\includegraphics[scale=.4]{UnCross.pdf}
\end{matrix}
$$
and all Reidemeister II and III moves and the following versions of Reidemeister I:
$$
 \begin{matrix}{}
\includegraphics[scale=.4]{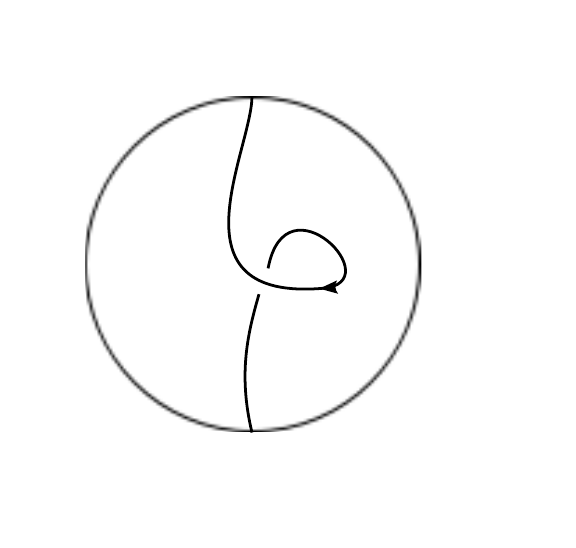}
\end{matrix}
=
 \begin{matrix}{}
\includegraphics[scale=.4]{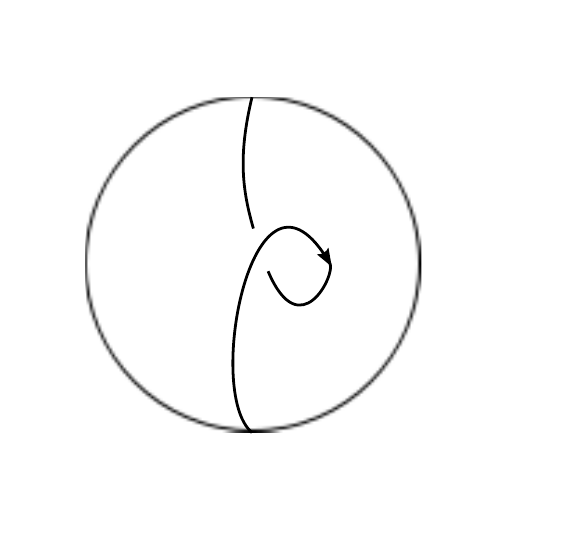}
\end{matrix}
= -q^{-1}
 \begin{matrix}{}
\includegraphics[scale=.4]{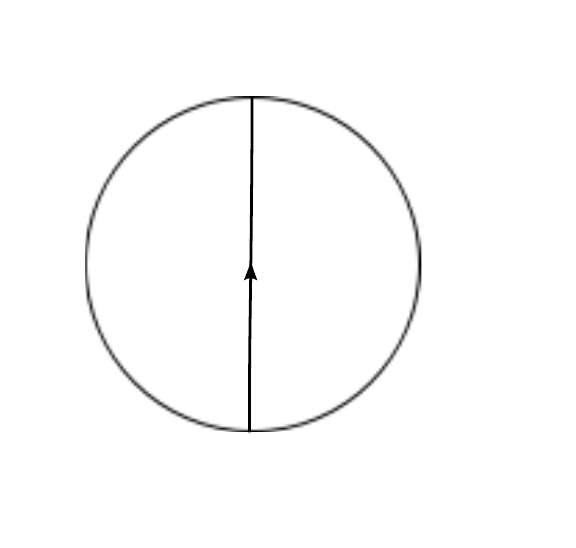}
\end{matrix}
$$
$$
 \begin{matrix}{}
\includegraphics[scale=.4]{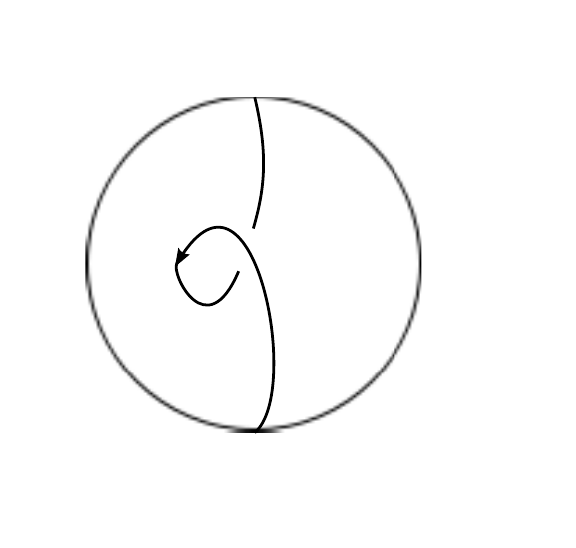}
\end{matrix}
=
 \begin{matrix}{}
\includegraphics[scale=.4]{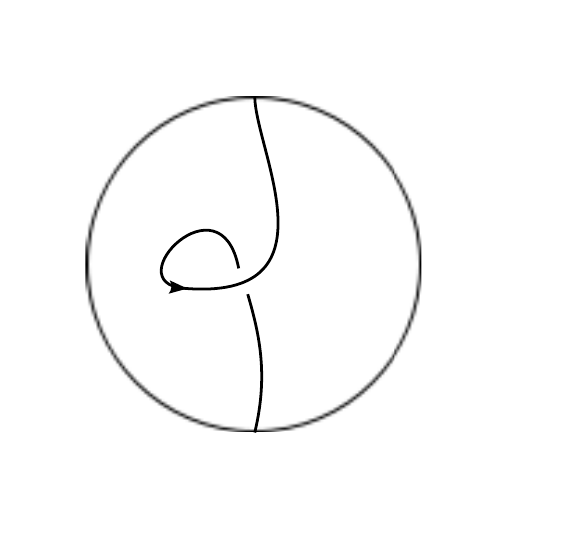}
\end{matrix}
= -q
 \begin{matrix}{}
\includegraphics[scale=.4]{R1nc.pdf}
\end{matrix}
.
$$
\end{lem}

Lemma \ref{LemRM} tells us that up to a normalizing coefficient, the resolution of crossings in Definition \ref{DeltaM} is a single variable tangle invariant that satisfies the skein relation for the Alexander polynomial.  Restricting ourselves to oriented knots and links in $T_2$,  the slight inconsistency with the first Reidemeister move can be dealt with by a simple normalizing coefficient.  Define the \emph{turning number} of a knot or link, $T$, to be the number of positively oriented loops minus the number of negatively oriented loops denoted $\tau(T)$.  This brings us to Bigelow's main result, Theorem 5.5 from \cite{Bigelow}.

\begin{thm}
\label{Thm5.2Big}
Where each strand has the same color, define
$$
\Delta'(T)=(-q)^{-\tau(T)}\Delta_m(T).
$$
If $T$ is an oriented tangle in $T_2$, then $\Delta'(T)$ is the Alexander polynomial of $\hat{T}$, the closure of $T$.
\end{thm}

The normalizing coefficient is a little more complicated for the multivariate version, but not much.  The \emph{rotation number} of a knot with color $q_i$, denoted rot$(q_i)$, is the change in the angle in the counterclockwise direction when following a component in a link projection in the direction of the orientation divided by $2\pi$.
\begin{thm}
\label{MainResult}
Define
$$
\begin{array}{c c c c}
\Delta_m':&T_2&\longrightarrow&P_2\\
&T&\longmapsto& N(T)\Delta_m(T)
\end{array}
$$
where $N(T)$ is the normalizing coefficient
$$
N(T)=\Big(\prod_{\textrm{colors, }q_i}(-q_i)^{-\textrm{rot}(q_i)}\Big)\Big/(q_l-q_l^{-1}),
$$
and $q_l$ is the color of the strand with endpoints on the boundary of the disk.  This map is the multivariate Alexander polynomial for a knot or link.  Moreover, $\Delta_m$ on higher tangle vector spaces gives a tangle invariant up to Reidemeister I.
\end{thm}

The subject of the next two sections will be proving Theorem $\ref{MainResult}$.  First in Section \ref{TI} we must show that $\Delta_m'$ respects the Reidemeister relations II and III exactly, and in $T_2$, Reidemeister I is satisfied up to the given normalizing coefficient.  Section \ref{MVAP} will be dedicated to showing that this map on $T_2$ satisfies the Murakami relations.

%%%%%%%%%%%%%%%%%%%%%%%%%%%%%%%%%%%%%%%%%%%%%%%%%
\section{We do have a tangle invariant}
\label{TI}

In order to prove Theorem \ref{MainResult}, we need to show that the Reidemeister moves are satisfied.  Recall the resolutions that we must check satisfy the Reidemeister relations:
%%%Pos Crossing Resolution%%%
$$
\begin{matrix}{}
\includegraphics[scale=.4]{PosCrossOU.pdf}
\end{matrix} =
q_o\begin{matrix}{}
\includegraphics[scale=.4]{DBasis1423.pdf}
\end{matrix} 
+
q_o\begin{matrix}{}
\includegraphics[scale=.4]{DBasis24cc.pdf}
\end{matrix} 
+
(q_u-q_u^{-1})\begin{matrix}{}
\includegraphics[scale=.4]{DBasis14.pdf}
\end{matrix} 
+
q_o^{-1}\begin{matrix}{}
\includegraphics[scale=.4]{DBasis13cc.pdf}
\end{matrix} 
-
q_o^{-1}\begin{matrix}{}
\includegraphics[scale=.4]{DBasis0.pdf}
\end{matrix} 
$$%%%Neg Crossing Resolution%%%
$$\begin{matrix}{}
\includegraphics[scale=.4]{NegCrossOU.pdf}
\end{matrix}=
q_o^{-1}\begin{matrix}{}
\includegraphics[scale=.4]{DBasis1423.pdf}
\end{matrix} 
+
q_o^{-1}\begin{matrix}{}
\includegraphics[scale=.4]{DBasis13cc.pdf}
\end{matrix} 
-
(q_u-q_u^{-1})\begin{matrix}{}
\includegraphics[scale=.4]{DBasis23.pdf}
\end{matrix} 
+
q_o\begin{matrix}{}
\includegraphics[scale=.4]{DBasis24cc.pdf}
\end{matrix} 
-q_o
\begin{matrix}{}
\includegraphics[scale=.4]{DBasis0.pdf}
\end{matrix} 
\hskip.03in.
$$

\begin{proof}
Showing that these resolutions satisfy the given versions of Reidemeister I involves only one strand, and this was covered in the proof in \cite{Bigelow}.  For Reidemeister II and III, we must show all different versions since we are dealing with oriented links.  We only need to show one version each of Reidemeister II and III, which we do in Lemmas \ref{LemR2} and \ref{LemR3}.  Together with the versions of Reidemeister I, these Reidemeister moves generate all of the other Reidemeister moves \cite{Polyak}.

\begin{lem}
\label{LemR2}
The following version of Reidemeister II is satisfied in the planar algebra $\PA$:
$$
\begin{matrix}{}
\includegraphics[scale=.4]{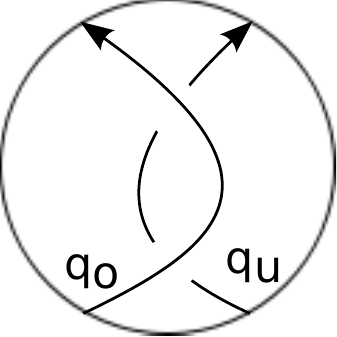}
\end{matrix} 
=
\begin{matrix}{}
\includegraphics[scale=.4]{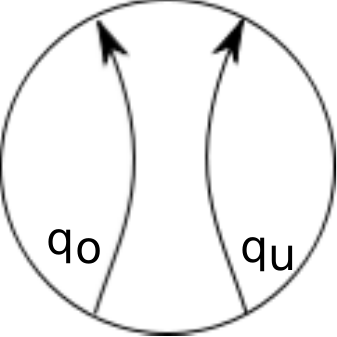}
\end{matrix} 
\hskip.03in.
$$
\end{lem}

\begin{proof}
We must show that the above diagram with crossings would evaluate the same way as if the strands were first slid past each other.  To show this, replace the positive and negative crossings with the linear combination of elements of $\PA$ described above.  We get twenty-five new diagrams, of which only six are not zero by Lemma \ref{LemBig3.4}.  Two diagrams cancel each other out immediately, and we are left with  
$$
\begin{matrix}{}
\includegraphics[scale=.4]{DBasis1423.pdf}
\end{matrix} 
+
\begin{matrix}{}
\includegraphics[scale=.4]{DBasis14.pdf}
\end{matrix} 
+
\begin{matrix}{}
\includegraphics[scale=.4]{DBasis23.pdf}
\end{matrix} 
+
\begin{matrix}{}
\includegraphics[scale=.4]{DBasis0.pdf}
\end{matrix}
$$
each with coefficient one.  After applying the definition of the dotted strand twice to the first term and once to the second, we see that the resolution simplifies to the same diagram with two uncrossed strands.
\end{proof}

The verification for the Reidemeister III move is more tedious, and we prove invariance under our algorithm in Lemma \ref{LemR3}.  

\begin{lem}
\label{LemR3}
The following versions of Reidemeister III are satisfied in the planar algebra $\PA$:
$$
\begin{matrix}{}
\includegraphics[scale=.4]{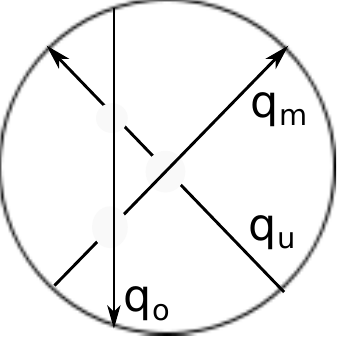}
\end{matrix}
=
\begin{matrix}{}
\includegraphics[scale=.4]{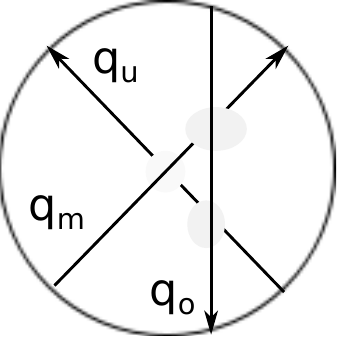}
\end{matrix}
\hskip.03in.
$$
\end{lem}

\begin{proof}
Resolving the crossings of the diagram in this version of Reidemeister III
$$
\begin{matrix}{}
\includegraphics[scale=.4]{R3posLeftOm3a.pdf}
\end{matrix}
\hskip.1in
\textrm{ and }
\hskip.1in
\begin{matrix}{}
\includegraphics[scale=.4]{R3posRightOm3a.pdf}
\end{matrix}
$$
gives a linear combination of 125 diagrams each.  All but fifteen from both linear combinations are zero.  After at most one application of the rotational relation to various diagrams in one linear combination, each of these fifteen diagrams appear in both resolutions with the same coefficients and cancel each other out.  
\end{proof}

After applying the result from Polyak, we conclude that this algorithm is invariant under all Reidemeister moves \cite{Polyak}.

\end{proof}

%%%%%%%%%%%%%%%%%%%%%%%%%%%%%%%%%%%%%%%%%%%%%%%%%
\section{The algorithm gives the multivariate Alexander polynomial}
\label{MVAP}
We will check that our algorithm gives the multivariate Alexander polynomial by checking the axioms from Murakami's 1993 paper ``A State Model for the Multi-variable Alexander Polynomial" \cite{JMur}.  All of the relations are straightforward except Murakami's third relation, which is a relation in the algebra of colored braids with three strands oriented upward.  We will define this relation using colored versions of $\sigma_1$ and $\sigma_2$ as positive generators and $e$ is the ``identity," or three strands with no crossings.  In the braid group multiplication from left to right should be interpreted as stacking from top to bottom.

\begin{defn}
Murakami's six axioms for a function $\Delta$ on knots and links to be the multivariate Alexander polynomial are:
\begin{enumerate}
\item The single variable skein relation for two strands with the same color.
\item $$(q_aq_b+q_a^{-1}q_b^{-1})\hskip.05in
\begin{matrix}{}
\includegraphics[scale=.4]{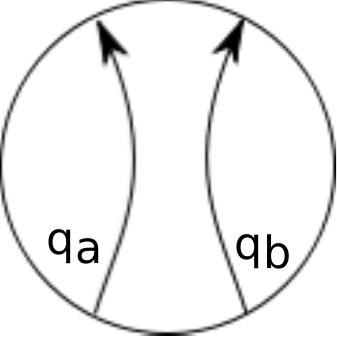}
\end{matrix}
=
\begin{matrix}{}
\includegraphics[scale=.4]{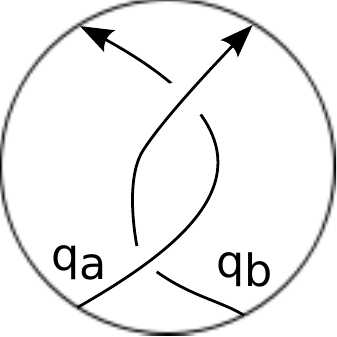}
\end{matrix}
+
\begin{matrix}{}
\includegraphics[scale=.4]{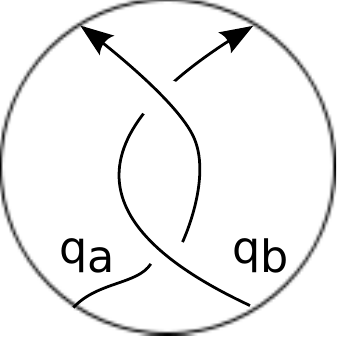}
\end{matrix}
$$

\item  Define $g_+(x)=x+x^{-1}$ and $g_-(x)=x-x^{-1}$.  From left to right along the bottom the colors of the strands in the algebra of colored braids are $q_a$, $q_b$, and $q_c$.  When we write $\Delta$ of a braid in brackets, we mean $\Delta$ of a link with that braid in it.  Each of the links in this relation is exactly the same everywhere except for locally differing by these braids.
$$
\begin{matrix}{}
g_+(q_c)g_-(q_b)\Delta([\sigma_1\sigma_2\sigma_2\sigma_1])-g_-(q_b)g_+(q_a)\Delta([\sigma_2\sigma_1\sigma_1\sigma_2])-\\
g_-(q_c^{-1}q_a)\big[\Delta([\sigma_1\sigma_1\sigma_2\sigma_2])+\Delta([\sigma_2\sigma_2\sigma_1\sigma_1])\big]+g_-(q_c^{-1}q_bq_a)g_+(q_a)\Delta([\sigma_2\sigma_2])\\
-g_+(q_c)g_-(q_cq_bq_a^{-1})\Delta([\sigma_1\sigma_1])-g_-(q_c^{-2}q_a^2)\Delta([e])\\
=\textrm{ ZERO }
\end{matrix}
$$
\item If $L$ is the trivial knot with color $q_a$, then $\Delta(L)=\frac{1}{q_a-q_a^{-1}}$.
\item
$$
\begin{matrix}{}
\includegraphics[scale=.4]{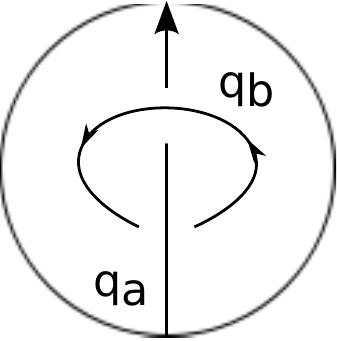}
\end{matrix}
=(q_a-q_a^{-1})\hskip.05in
\begin{matrix}{}
\includegraphics[scale=.4]{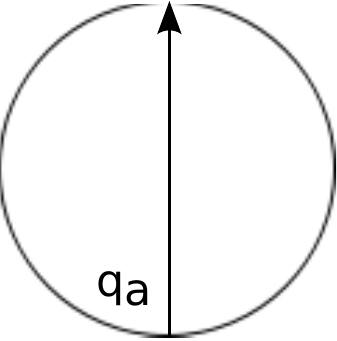}
\end{matrix}
$$
\item If $L$ is the split union of a link and trivial knot, then $\Delta(L)$ is zero.
\end{enumerate}
\end{defn}

\begin{proof}
These axioms are equally applicable to a 1-tangle, which we can write as an almost closed braid with only the stand having endpoints on the boundary of the disk left unclosed.  Indeed, the first three relations do not involve the portion of any strand connecting the top of the braid to the bottom of the braid.  Relations 4-6 can be applied to any disjoint union of connected sums of Hopf links written as a 1-tangle to reduce the tangle to zero or a polynomial times the unclosed strand.  So the multivariate Alexander polynomial of any 1-tangle is evaluable using Murakami's relations, and the algorithm is independent of the choice of how to write the knot or link as a 1-tangle \cite{JMur}.  

It is straightforward to check most of these relations.  Relation 1 is the single variable skein relation shown by Bigelow \cite{Bigelow}.  Relation 4 is satisfied by the normalizing coefficient.  
Relation 6 comes directly from the first relation of the planar algebra that a closed loop is zero.  Relations 2 and 5 are no more difficult than the second Reidemeister move and can be checked by hand.

Proving Murakami's third relation is significantly harder.  We must define a representation on $\mathbb{C}B_3$, the algebra of colored braids with three strands.  Let $b$ be a linear combination of colored braids in $\mathbb{C}B_3$, and $L_b$ will represent a linear combination of knots or links that are identical everywhere except where they differ locally according to $b$.  We must define a representation, $\phi$, so that $\phi(b)=0$ implies that $\Delta_m'(L_b)=0$.  It might be worth noting a difference in convention here.  Murakami has all of his braids oriented downward, whereas ours are oriented upward.  We hope pointing this out directly will avoid some frustration over convention.

\begin{defn}
Let $\mathbb{C}B_3$ be the algebra of colored braids.  That is to say braids where each strand has a color or label from $\{q_a,q_ b,q_ c\}$.  So there are now six versions of $\sigma_1$:
$$
\sigma_1^1=\begin{matrix}{}
\includegraphics[scale=.3]{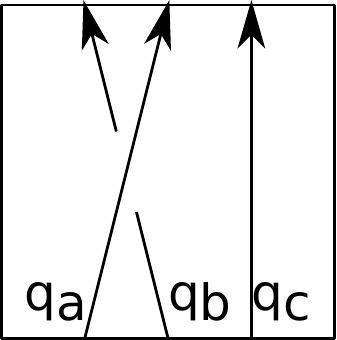}
\end{matrix}\hskip.01in,\hskip.05in
\sigma_1^2=\begin{matrix}{}
\includegraphics[scale=.3]{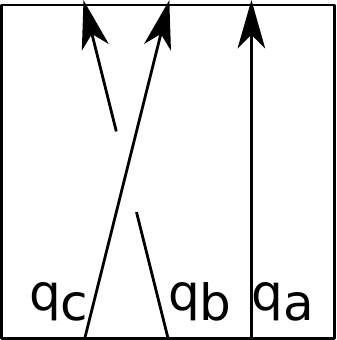}
\end{matrix}\hskip.01in,\hskip.05in
\sigma_1^3=\begin{matrix}{}
\includegraphics[scale=.3]{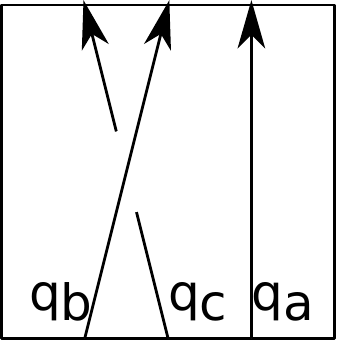}
\end{matrix}\hskip.01in,\hskip.05in
\sigma_1^4=\begin{matrix}{}
\includegraphics[scale=.3]{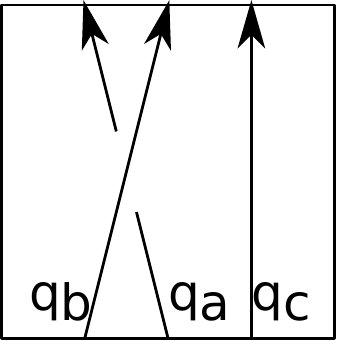}
\end{matrix}\hskip.01in,\hskip.05in
\sigma_1^5=\begin{matrix}{}
\includegraphics[scale=.3]{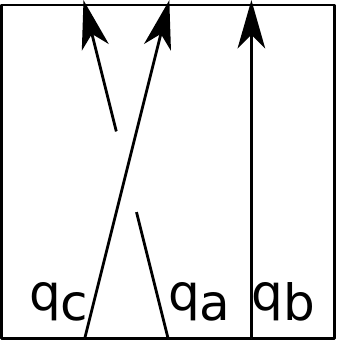}
\end{matrix}\hskip.01in,\hskip.05in
\sigma_1^6=\begin{matrix}{}
\includegraphics[scale=.3]{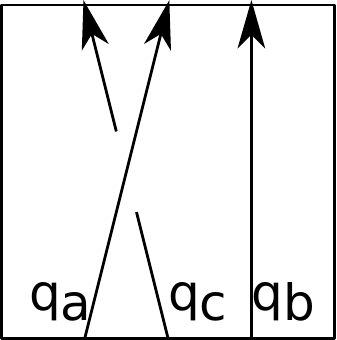}
\end{matrix}
\hskip.03in.
$$

\noindent The second generator follows the same conventions; $\sigma_2^1$ is labeled $q_a$, $q_b$, and $q_c$ from left to right across the bottom, $\sigma_2^2$ is labeled $q_c$, $q_b$, $q_a$, and so on.

Formally, you can stack any diagrams on top of one another, but the braid relations only exist for braids with a consistent coloring throughout the strand.  This includes the ``identity," $e$. 
\end{defn}

Let $P_6$ be the sixth vector space of $\PA$.  Define $V$ to be the subspace of $P_6$ generated by basis elements:
$$
v_1=\begin{matrix}{}
\includegraphics[scale=.4]{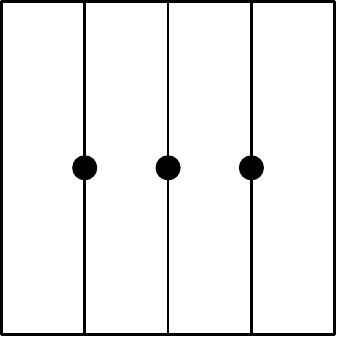}
\end{matrix}\hskip.02in,\hskip.1in
v_2=\begin{matrix}{}
\includegraphics[scale=.4]{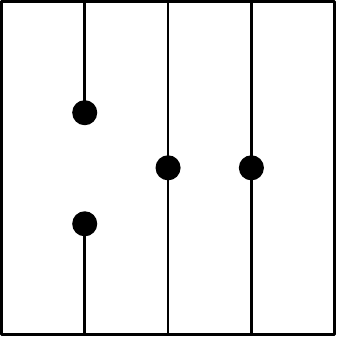}
\end{matrix}\hskip.02in,\hskip.1in
v_3=\begin{matrix}{}
\includegraphics[scale=.4]{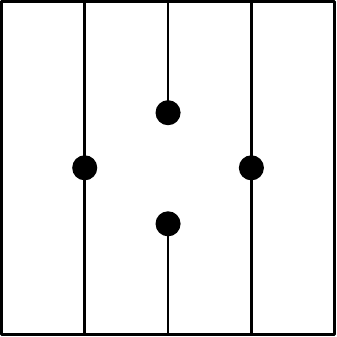}
\end{matrix}\hskip.02in,\hskip.1in
v_4=\begin{matrix}{}
\includegraphics[scale=.4]{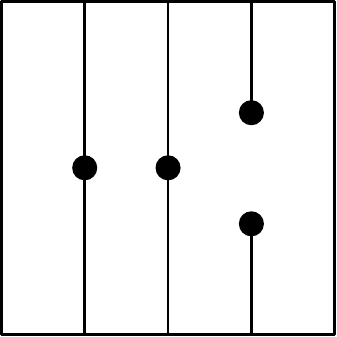}
\end{matrix}\hskip.02in,\hskip.1in
$$

$$
v_5=\begin{matrix}{}
\includegraphics[scale=.4]{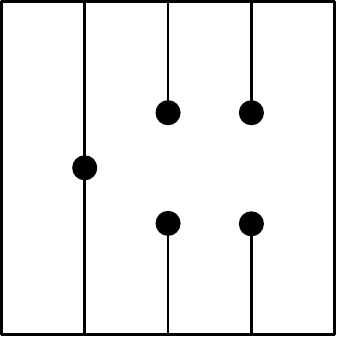}
\end{matrix}\hskip.02in,\hskip.1in
v_6=\begin{matrix}{}
\includegraphics[scale=.4]{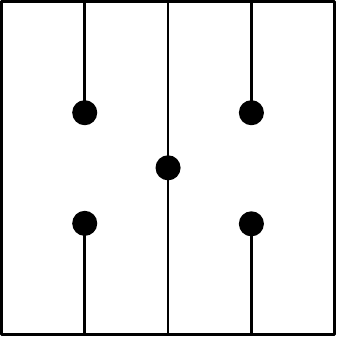}
\end{matrix}\hskip.02in,\hskip.1in
v_7=\begin{matrix}{}
\includegraphics[scale=.4]{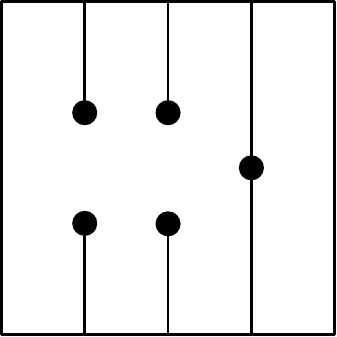}
\end{matrix}\hskip.02in,\hskip.1in
v_8=\begin{matrix}{}
\includegraphics[scale=.4]{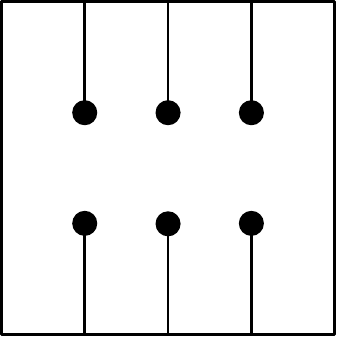}
\end{matrix}
\hskip.03in.
\hskip.1in
$$

The representation $\phi$ will be defined by resolving crossings of a braid in $\mathbb{C}B_3$ and sending it into $P_6$ and then into $M_8(V,\mathbb{C})$.  Call the map from $\mathbb{C}B_3$ into $P_6$ that resolves crossings, $\rho$.  For all $b\in\mathbb{C}B_3$, $\rho(v_ibv_j)$ is equal to a coefficient, $b_{i,j}$, times a single diagram.  This is because with the given relations in $P_6$, there is only one way of connecting the strands of $v_i$ and $v_j$ by Lemma \ref{LemBig3.4}.  The representation $\phi$ is defined by:
$$
(\phi(b))_{i,j}=b_{i,j}.
$$

Since $b_{i,j}=0$ if $v_i$ and $v_j$ have a different number of dotted strands also by Lemma \ref{LemBig3.4}, the matrix $\phi(b)$ is a nice block matrix.  The matrix for $\sigma_1^1$ is
$$
\left(
\begin{array}{c c c c c c c c}
q_a&0&0&0&0&0&0&0\\
0&0&q_a^{-1}&0&0&0&0&0\\
0&q_a&q_b-q_b^{-1}&0&0&0&0&0\\
0&0&0&q_a&0&0&0&0\\
0&0&0&0&q_b-q_b^{-1}&q_a&0&0\\
0&0&0&0&q_a^{-1}&0&0&0\\
0&0&0&0&0&0&-q_a^{-1}&0\\
0&0&0&0&0&0&0&-q_a^{-1}\\
\end{array}
\right).
$$
All of the versions of $\sigma_1$ are of the same form.  Changing the coloring of the strands only permutes the variables.  So the matrix $\phi(\sigma_1^2)$ is
$$
\left(
\begin{array}{c c c c c c c c}
q_c&0&0&0&0&0&0&0\\
0&0&q_c^{-1}&0&0&0&0&0\\
0&q_c&q_b-q_b^{-1}&0&0&0&0&0\\
0&0&0&q_c&0&0&0&0\\
0&0&0&0&q_b-q_b^{-1}&q_c&0&0\\
0&0&0&0&q_c^{-1}&0&0&0\\
0&0&0&0&0&0&-q_c^{-1}&0\\
0&0&0&0&0&0&0&-q_c^{-1}\\
\end{array}
\right).
$$\\
The other versions of the generator $\sigma_2$ are calculated in the same way, and $\sigma_2^1$, which is labeled along the bottom with strands colored $q_a$, $q_b$, and $q_c$, is
$$
\left(
\begin{array}{c c c c c c c c}
q_b&0&0&0&0&0&0&0\\
0&q_b&0&0&0&0&0&0\\
0&0&0&q_b^{-1}&0&0&0&0\\
0&0&q_b&q_c-q_c^{-1}&0&0&0&0\\
0&0&0&0&-q_b^{-1}&0&0&0\\
0&0&0&0&0&q_c-q_c^{-1}&q_b&0\\
0&0&0&0&0&q_b^{-1}&0&0\\
0&0&0&0&0&0&0&-q_b^{-1}\\
\end{array}
\right).
$$

We must show $(\phi(uv))_{i,j}=(\phi(u)\phi(v))_{i,j}$ for $u$, $v\in\mathbb{C}B_3$.  To do this, first note that $\rho(e)$ in $P_6$ is equal to $\Sigma_{k=1}^8v_k$ and $v_k^2=v_k$ for all $k$.  Then the $i,j$ entry of $\phi(uv)$ is equal to the coefficient of $v_iu(id)vv_j=\Sigma_{k=1}^8v_iuv_kvv_j=\Sigma_{k=1}^8v_iuv_k^2vv_j$.  So $(\phi(uv))_{i,j}=\Sigma_{k=1}^8(\phi(u))_{i,k}(\phi(v))_{k,j}$, or the $i$th row of $\phi(u)$ dotted with the $j$th column of $\phi(v)$, which is $(\phi(u)\phi(v))_{i,j}$.  So the proposed representation of $\mathbb{C}B_3$ is multiplicative.

To finish showing that we have a representation of of $\mathbb{C}B_3$, we need to check the third Reidemeister move holds.  In the uncolored braid group, this relation is written $\sigma_1\sigma_2\sigma_1=\sigma_2\sigma_1\sigma_2$.  We must make sure to use the appropriate versions of the generators so that the two braids have a consistent coloring along the strands.  Labeling the bottom of the strands from left to right $q_a$, $q_b$, and $q_c$, we must verify that $\phi(\sigma_1^3\sigma_2^4\sigma_1^1)=\phi(\sigma_2^5\sigma_1^6\sigma_2^1)$.  Calculating the appropriate versions of $\sigma_1$ and $\sigma_2$, the third Reidemeister move is easy to check in Mathematica.  This shows that we have a representation of the colored braid group.

Suppose $L_b$ is a linear combination of knots or links that are identical except locally where they differ by the terms in a linear combination of braids in $\mathbb{C}B_3$, $b$.  Further suppose that  $\phi(b)=0$.  That $\Delta_m'(L_b)=0$ follows from the fact $\rho(e)=\Sigma_{k=1}^8v_k$:

$$
\Delta_m'(L_b)=
\Delta_m'\left( \begin{matrix}{}
\includegraphics[scale=.35]{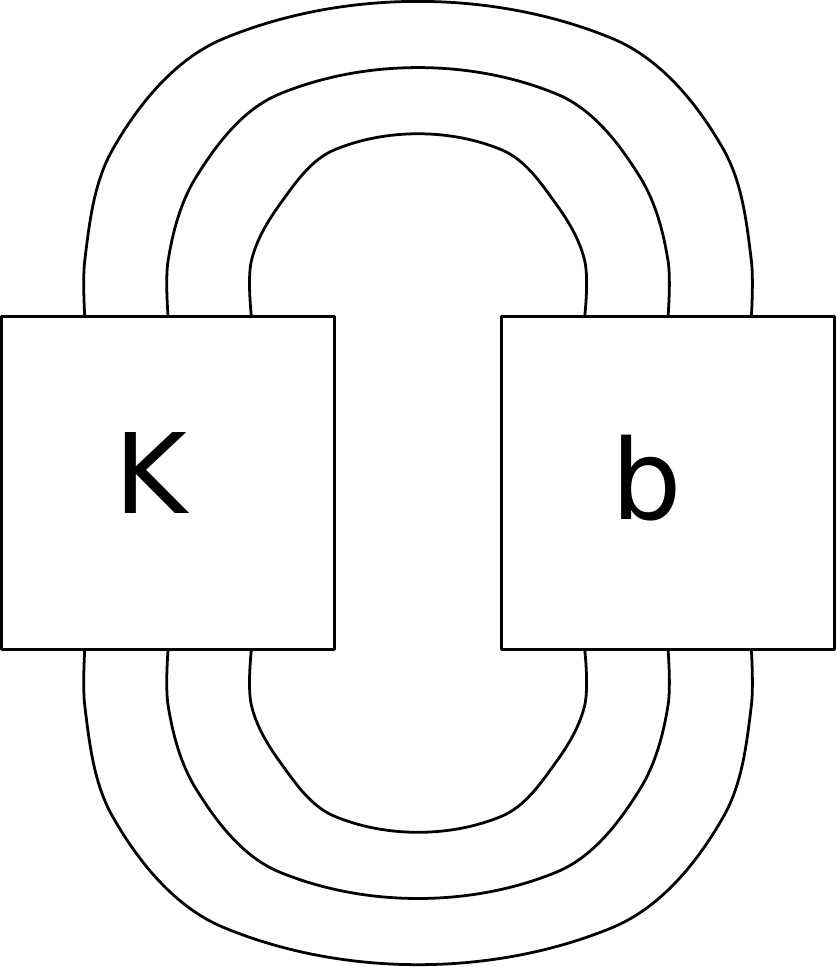}
\end{matrix}
\right)
=
\Delta_m'\left( \begin{matrix}{}
\includegraphics[scale=.35]{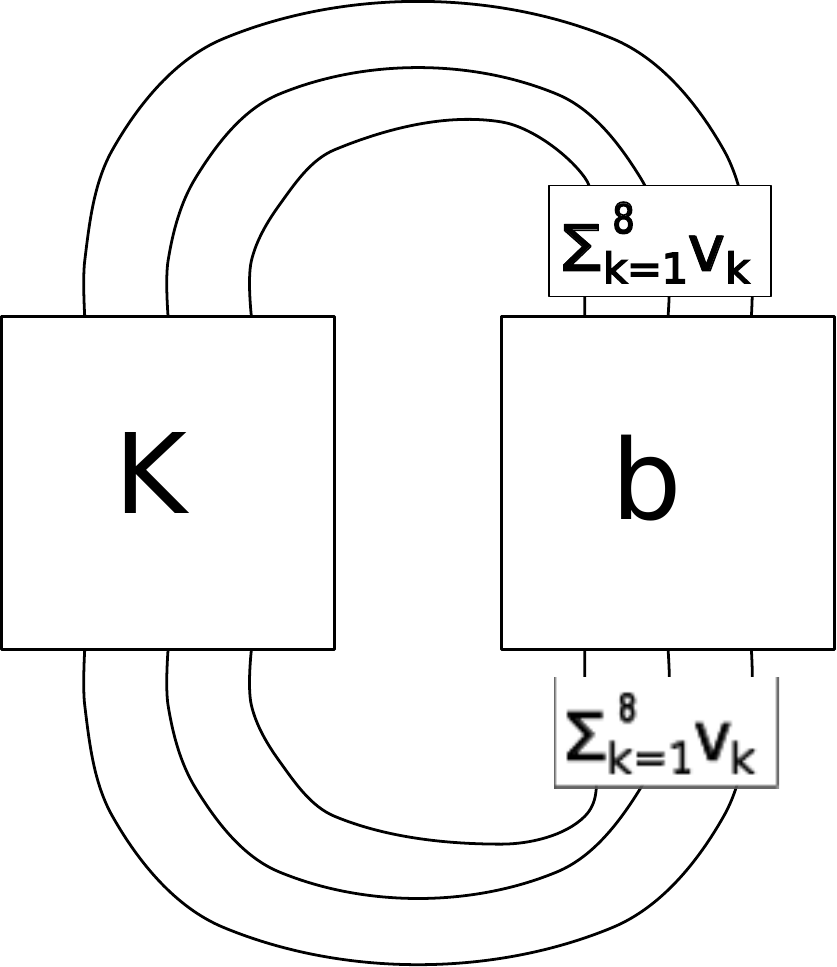}
\end{matrix}
\right)
=
\Delta_m'(0)
=0
\hskip.03in.
$$

We have a representation of $\mathbb{C}B_3$ with the desired property.  Checking Murakami's third relation is now an easy problem for Mathematica to do, and we see that the representation of the linear combination of braids in Murakami's third relation is zero.  I will include this Mathematica notebook on my website http://www.math.ucsb.edu/$\sim$kgracekennedy/.  We have shown the Murakami relations hold and that we have an algorithm to calculate the multivariate Alexander polynomial of a link.

\end{proof}

%%%%%%%%%%%%%%%%%%%%%%%%%%%%%%%%%%%%%%%%%%%%%%%%%
\section{Conclusion}
\label{Conc}
The resolutions of crossings defined in this paper give a new way to calculate the multivariate Alexander polynomial of a link. The algorithm generalizes nicely to a tangle invariant up to Reidemeister I.  The tangle invariant is not a single polynomial but rather a linear combination of diagrams with coefficients that are polynomials.  There are several open questions about this invariant and the planar algebra, such as is this the same tangle invariant as the one given by Archibald in \cite{ArchThesis}.  There is strong evidence to suggest that it is.

Anyone interested in planar algebras would probably like to know if this has any relation to subfactor theory.  The planar algebra $\PA$ is not a subfactor planar algebra, but it is closely related to the Temperley-Lieb planar algebra.  For example, if we take a diagram in $\PA$ with no dots, we can ``thicken" the strands.  Broken strands  become shaded caps and cups, and through strands, or strands with one endpoint on the top and one on the bottom, become two strands that bound a shaded region.  For instance,

$$
\begin{matrix}{}
\includegraphics[scale=.4]{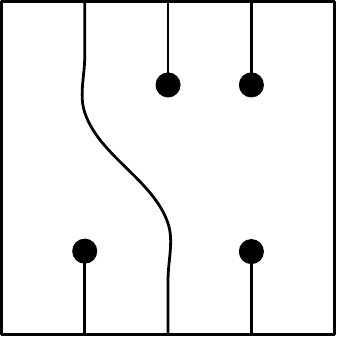}
\end{matrix}
\hskip.1in\textrm{ becomes }\hskip.1in
\begin{matrix}
\includegraphics[scale=.4]{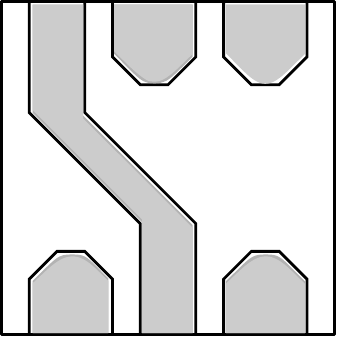}\hskip.1in.
\end{matrix}
$$
Diagrams with $n$ boundary points on the top and bottom with only broken strands and through strands form a subalgebra of the Temperley-Lieb algebra with $n$ strands \cite{Evans}.  These diagrams came up in the recent work of Jones and Penneys on infinite index subfactors.  These subalgebras of the Temperley-Lieb algebras always appear injectively in the standard invariant of finite and infinite index subfactors \cite{Dave}.  It is still unknown if there exists an infinite index subfactor for which the standard invariant is exactly these subalgebras \cite{Dave}, \cite{DaveInfInd}.

\bibliography{MAP_PA_Kennedy_Arxiv}

\end{document}